\documentclass{amsart}
\usepackage{amssymb,amscd,amsthm,verbatim}
\input epsf                 

\newtheorem{proposition}{Proposition}[section]
\newtheorem{theorem}[proposition]{Theorem}

\newtheorem{lemma}[proposition]{Lemma}
\newtheorem{prop}[proposition]{Proposition}
\newtheorem{cor}[proposition]{Corollary}

\theoremstyle{definition}
\newtheorem{example}[proposition]{Example}

\theoremstyle{remark}
\newtheorem{remark}[proposition]{Remark}

\numberwithin{equation}{section}

\newcommand{\integers}{\mathbb Z}

\newcommand{\C}{\operatorname{C}}
\newcommand{\SC}{\operatorname{SC}}
\newcommand{\E}{\operatorname{E}}
\newcommand{\NC}{\operatorname{NC}}
\newcommand{\MC}{\operatorname{MC}}
\newcommand{\TW}{\operatorname{TW}}

\newcommand{\covered}{{\,\,<\!\!\!\!\cdot\,\,\,}}
\newcommand{\rank}{\mathrm{rank}}
\newcommand{\set}[1]{{\left\lbrace #1 \right\rbrace}}
\newcommand{\coef}[2]{{\left.\left\langle #1\right| #2  \right\rangle}}
\newcommand{\br}[1]{\langle #1 \rangle}
\newcommand{\Cat}{\operatorname{Cat}}

\begin{document}
\title{Chains in the noncrossing partition lattice}

\author{Nathan Reading}
\address{Department of Mathematics\\
North Carolina State University\\
Raleigh, NC 27695}
\thanks{The author was partially supported by NSF grant DMS-0202430.}
\email{nathan\_reading@ncsu.edu}
\urladdr{http://www4.ncsu.edu/\textasciitilde nreadin}
%\subjclass[2000]{Primary 20F55; Secondary 05A18, 05E15}
%\keywords{absolute order, Coxeter element, Coxeter group, generalized cluster complex, noncrossing partition, parking function}

\begin{abstract}
%sinceV1:  Rewrote abstract to avoid "false advertising" (comment by Stembridge).
We establish recursions counting various classes of chains in the noncrossing partition lattice of a finite Coxeter group.
The recursions specialize a general relation which is proven uniformly (i.e. without appealing to the classification of finite Coxeter groups) using basic facts about noncrossing partitions.
We solve these recursions for each finite Coxeter group in the classification.
Among other results, we obtain a simpler proof of a known uniform formula for the number of maximal chains of noncrossing partitions and a new uniform formula for the number of edges in the noncrossing partition lattice.
All of our results extend to the $m$-divisible noncrossing partition lattice.
\end{abstract} 

\maketitle

\section{Introduction}
The lattice of noncrossing partitions was defined and studied in 1972 by Kreweras~\cite{Kreweras}.
For surveys of results on this lattice and on its mathematical applications, see~\cite{McCammond,Simion}.
Through the results of \cite{Bessis,Biane1,BWKpi,Rei}, the noncrossing partition lattice was recognized as a special case of a construction valid for an arbitrary finite Coxeter group.
The notation~$L_W$ will stand for the noncrossing partition lattice of a finite Coxeter group~$W$.
In particular, when~$W$ is the symmetric group,~$L_W$ is the usual noncrossing partition lattice.
%sinceV1:  added the Stanley reference in the following sentence.
Detailed enumeration of chains in~$L_W$ for various Coxeter groups $W$ has been carried out in \cite{Armstrong,Ath-Rei,Biane2,Chapoton,Edelman,Kr-M,Kreweras,Rei,StNC}.

%sinceV1:  Rewrote next paragraph to avoid "false advertising" (comment by Stembridge).
The key result of this paper is a formula relating certain rank-selected chain numbers for $L_W$ to chain enumerations arising in parabolic subgroups.
%sinceV1:  Rewrote next sentence thanks to Christos' comment.
%sinceSIAM: changed C_{(...)}(W) to \C_{(...)}(W) throughout
For a sequence $(j_1,j_2,\ldots,j_{k+1})$ summing to $n=\rank(W)$, let $\C_{(j_1,j_2,\ldots,j_{k+1})}(W)$ count multichains $x_1\le x_2\le\cdots\le x_k$ in~$L_W$ with \mbox{$\ell_T(x_1)=j_1$}, $\ell_T(x_k)=n-j_{k+1}$ and $\ell_T(x_i)=\ell_T(x_{i-1})+j_i$ for $i=2,\ldots,k$.
Here $\ell_T$ is the rank function of~$L_W$.
%sinceSIAM:  Replaced all "-" signs used for complementation with "\setminus"
For each simple reflection $s\in S$, let $W_{\br{s}}$ denote the parabolic subgroup generated by $S\setminus\set{s}$.
The Coxeter number $h$ is the order of a Coxeter element of~$W$.
%sinceV1:  Rewrote next sentence thanks to Christos' comment.
%referee:  changed j_k to j_{k+1}
%frank:  inserted comma before \widehat
The sequence $(j_1,j_2,\ldots,\widehat{j_i},\ldots,j_{k+1})$ is obtained by deleting $j_i$ from $(j_1,j_2,\ldots,j_{k+1})$.
\begin{theorem}\label{jump chains}
If $(W,S)$ is a finite irreducible Coxeter system and $j_i=1$ then
%referee:  changed j_{k+i} to j_{k+1}
%frank:  inserted comma before \widehat
\[\C_{(j_1,j_2,\ldots,j_{k+1})}(W)\,=\,\frac{h}{2}\,\sum_{s\in S}\C_{(j_1,j_2,\ldots,\widehat{j_i},\ldots,j_{k+1})}(W_{\br{s}}).\]
\end{theorem}

The theorem is proved uniformly in Section~\ref{proof} by a method similar to that used by Fomin and Zelevinsky to prove a recursive formula \cite[Proposition~3.7]{ga} counting the facets of the cluster complex.
In that context, one ``rotates'' a root by a modified Coxeter element (of order $h+2$) until one obtains the negative of a simple root.
This allows one to pass to a parabolic subgroup.
Here, we rotate a reflection by an unmodified Coxeter element (of order $h$) until we obtain a simple reflection, which allows us to pass to a parabolic subgroup.

%sinceV1: altered following sentence to make it read better.
Theorem~\ref{jump chains} is a broad generalization of what appears to be the only nontrivial enumerative fact about~$L_W$  previously known to have a uniform proof:
the formula $nh/2$ for the number of atoms (or coatoms) of~$L_W$.
%sinceV1:  altered following sentence (thanks to Stembridge's comment).
There are, however, uniform bijections to other sets, namely clusters \cite{ABMW,IngThom,sortable} and sortable elements \cite{IngThom,sortable}, but no uniform proof is known for the enumeration of these other sets.
(The results of~\cite{IngThom} apply only to the crystallographic case.
The proofs in~\cite{sortable} are made uniform by the results of~\cite{typefree}.)
%sinceV1:  altered following sentence (thanks to Stembridge's comment).
There is also a uniform determination \cite[Corollary~4.4]{ABW} of the M\"{o}bius function of~$L_W$ in terms of positive clusters, but no uniform proof is known for the enumeration of positive clusters.

In Section~\ref{app sec}, we specialize Theorem~\ref{jump chains} to provide recursions for some important classes of chains.
In some cases, the recursion leads to a uniform formula.
However, even in those cases, deriving the formula from the recursion requires a type-by-type approach.
We now briefly summarize the results obtained.

We first consider the number $\MC(W)$ of maximal chains in~$L_W$.
We obtain a recursion for $\MC(W)$ and, by solving the recursion type-by-type, a uniform formula for $\MC(W)$.
As pointed out in~\cite{Chapoton}, this uniform formula follows from previous type-by-type determinations of the zeta polynomial of~$L_W$.
In the exceptional types, the recursion on $\MC(W)$ can be solved without a computer, thus providing the first verification of the formula for $\MC(W)$ without brute-force computer counting.

We next give a recursion on the number of reduced words (in the alphabet of reflections) for elements of~$L_W$.
The recursion implies a relationship between the number of such words and the face numbers of the generalized cluster complexes of~\cite{gcccc}.
Another specialization of Theorem~\ref{jump chains} leads to a uniform formula, which appears to be new, for the number of edges in~$L_W$.
More generally, we consider saturated chains in~$L_W$ of a fixed length.
The number of such chains appears to exhibit the same odd behavior observed in~\cite{gcccc} for the $f$- and $h$-numbers of generalized cluster complexes.

Theorem~\ref{jump chains} and its corollaries are statements about $h$-fold symmetry.
It is intriguing that one of the corollaries can be obtained, by taking leading coefficients, from a similar statement about the $(mh+2)$-fold symmetry of a generalized cluster complex, as explained in Section~\ref{app sec}.

We conclude the paper with a brief discussion, in Section~\ref{m sec}, of generalizations to $m$-divisible noncrossing partitions.

%sinceV1: changed title of this section.
\section{Proof of the main theorem}\label{proof}
In this section, we define~$L_W$  and gather the simple facts about Coxeter groups and noncrossing partitions that are necessary to prove Theorem~\ref{jump chains}.
We then prove the theorem and comment on the case where $W$ is reducible.
Much of the background material on~$L_W$ is due to Armstrong~\cite{Armstrong}, Bessis~\cite{Bessis} and Brady and Watt~\cite{BWorth,BWlattice}.

We assume basic background on Coxeter groups and root systems, which is found, for example, in~\cite{Bj-Br,Bourbaki,Humphreys}.
Let $(W,S)$ be a finite Coxeter system of rank~$n$.
%sinceSIAM:  changed "elements of~$W$ and their action on~$V$" to "an element of~$W$ and its action on~$V$."
We fix a representation of~$W$ as a real reflection group acting on a Euclidean space~$V$ and make no distinction between an element of~$W$ and its action on~$V$.
%sinceSIAM:  removed the sentence that was here and put it after the first sentence of the next paragraph.

Let $T$ be the set of reflections in~$W$.
%sinceSIAM:  put the sentence from the previous paragraph here.
For any reflection $t\in T$ let $H_t$ be the reflecting hyperplane associated to $t$.
Any element $w\in W$ can be written as a \emph{$T$-word}---a word in the alphabet $T$.
A \emph{reduced} $T$-word for $w$ is a $T$-word which has minimal length among all $T$-words for $w$.
The \emph{absolute length} of $w$ is the length of a reduced $T$-word for $w$.
This should not be confused with the more common notion of length in~$W$: the length of a reduced word for $w$ in the alphabet~$S$.
The \emph{absolute order} on~$W$ sets $u\le v$ if and only if $u$ has a reduced $T$-word which is a prefix of some reduced $T$-word for $v$.
Note that since $T$ is fixed as a set by conjugation,  for any $u\le v$ and any $w\in W$, the interval $[u,v]$ in the absolute order is isomorphic to the interval $[wuw^{-1},wvw^{-1}]$.

Suppose $t_1t_2\cdots t_k$ is a reduced $T$-word for $w$.
Then, for any $i\in[k-1]$, another reduced $T$-word for $w$ can be obtained by replacing $t_i$ with $t_{i+1}$ and $t_{i+1}$ with the reflection $(t_{i+1}t_it_{i+1})$, while leaving all other letters of the word unchanged.
Similarly, for $i\in[2,k]$, one can replace $t_i$ with $t_{i-1}$ and $t_{i-1}$ with $(t_{i-1}t_it_{i-1})$.
This implies that $u\le v$ if and only if $u$ has a reduced $T$-word which is a subword of some reduced $T$-word for $v$.
Furthermore $u\le v$ if and only if $u$ has a reduced $T$-word which is a postfix of some reduced $T$-word for $v$.

A Coxeter element~$c$ is any element of~$W$ of the form $c=s_1s_2\cdots s_n$ with each element of~$S$ occurring exactly once.
The order of~$c$ is the \emph{Coxeter number} $h$.
The primary object of study in this paper is the interval $[1,c]$ in the absolute order, often called the noncrossing partition lattice and denoted here by~$L_W$.
Any two Coxeter elements are conjugate in~$W,$ so the isomorphism type of~$L_W$ does not depend on the choice of Coxeter element~$c$.
The fact that~$L_W$ is a lattice was given a uniform proof in~\cite{BWlattice} and later, with slightly less generality in~\cite{IngThom}.

The Coxeter diagram of~$W$ is a tree and thus a bipartite graph.
Let $S=S_+\cup S_-$ be a bipartition of the diagram.
Define involutions 
\[c_+=\prod_{s\in S_+}s\ \ \mbox{ and }\ \ c_-=\prod_{s\in S_-}s\]
so that $c=c_-c_+$ is a Coxeter element with $c_+cc_+=c_-cc_-=c_+c_-=c^{-1}$.
Thus conjugation by $c_+$ and conjugation by $c_-$ are isomorphisms from $[1,c]$ to $[1,c^{-1}]$.
%sinceSIAM:  changed to d's here too.
Let $c_+^{\br{d}}$ be the $|d|$-fold product $c_{(-1)^d}\cdots c_+ c_-c_+$ if $d\ge 0$ or 
$c_+c_-c_+\cdots c_{(-1)^d}$ if $d<0$.
We have $c_+^{\br{d}}c_+^{\br{-d}}=1$ for any $d\in\integers$ and furthermore $c_+^{\br{2h}}=c^h=1$.

The key to the proof of Theorem~\ref{jump chains} is a result from Steinberg's 1959 paper~\cite{Steinberg} on finite reflection groups.
(See also \cite[Sections 3.16--3.20]{Humphreys} or \cite[Section V.6.2]{Bourbaki}.)
\begin{prop}\label{orbits}
{\rm (Steinberg)}\,\,
Let $(W,S)$ be an irreducible finite Coxeter system.
The orbit of any reflection under the conjugation action of the dihedral group $\br{c_+,c_-}$ either:
\begin{enumerate}
\item[(i) ]has $h/2$ elements and intersects~$S$ in a single element, or
\item[(ii) ]has $h$ elements and intersects~$S$ in a two-element set.
\end{enumerate}
\end{prop}

The proof of Theorem~\ref{jump chains} uses some fundamental facts about absolute order which we now quote as Theorem~\ref{BWmain}.
A clean seven-page exposition (with complete proofs) of these results can be obtained by reading Brady and Watt's paper~\cite{BWorth} followed by Section 2 of their paper~\cite{BWKpi}.
We phrase these properties in terms of fixed spaces $F_w=\mbox{kernel}(w-I)$ of elements $w\in W,$ rather than the  \emph{moved spaces} of~\cite{BWorth}.
This change is harmless because the moved space is the orthogonal complement of~$F_w$.
Note that $H_t=F_t$ for each $t\in T$.

\begin{theorem}\label{BWmain}
{\rm (Brady and Watt)}
\begin{enumerate}
\item[(i) ]If $t_1t_2\cdots t_k$ is a reduced $T$-word for $w\in W$ then $F_w=H_{t_1}\cap H_{t_2}\cap\cdots\cap H_{t_k}$.
\item[(ii) ] If $x,y\in [1,c]$ then $x\le y$ if and only if $F_y\subseteq F_x$.
\item[(iii) ] If $t\in T$ then $t\le c$.
\end{enumerate}
\end{theorem}

One useful consequence of Theorem~\ref{BWmain} is the following lemma (cf. \cite[Lemma 2.3]{BWKpi}).
Recall that $W_{\br{s}}$ is the parabolic subgroup generated by $S\setminus\set{s}$ and let $c'$ be the Coxeter element for $W_{\br{s}}$ obtained by deleting~$s$ from the defining word for~$c$.
%referee:  Revised the material below, through the statement and proof of Lemma~\ref{para}.
The subword characterization of absolute order implies that $c'<c$, so $[1,c']\subset[1,c]$.
We write $[u,v]_{\br{s}}$ for an interval in the absolute order on $W_{\br{s}}$ and continue to write $[u,v]$ for an interval in the absolute order on $W$.

\begin{lemma}\label{para}
The inclusion $W_{\br{s}}\hookrightarrow W$ restricts to an isomorphism from $[1,c']_{\br{s}}$ to $[1,c']$.
\end{lemma}
\begin{proof}
Theorem~\ref{BWmain}(ii) implies that the inclusion $W_{\br{s}}\hookrightarrow W$ restricts to an isomorphism from $[1,c']_{\br{s}}$ to its image and, furthermore, that this image is contained in $[1,c']$.
The parabolic subgroup $W_{\br{s}}$ is the set of elements of~$W$ which fix the subspace $\cap_{s'\in\br{s}}H_{s'}$.
But $\cap_{s'\in\br{s}}H_{s'}=F_{c'}$ by Theorem~\ref{BWmain}(i), so Theorem~\ref{BWmain}(ii) implies that $[1,c']\subseteq W_{\br{s}}$.
Thus the inclusion of $[1,c']_{\br{s}}$ in $W$ is the entire interval $[1,c']$.
\end{proof}

To simplify the proof of Theorem~\ref{jump chains}, we employ a basic result about counting multichains in~$L_W$.
(Cf. \cite[Lemma~3.1.2]{Armstrong}.)
\begin{prop}\label{permute}
If $(j'_1,\ldots,j'_{k+1})$ is a permutation of $(j_1,\ldots,j_{k+1})$ then 
\[\C_{(j'_1,\ldots,j'_{k+1})}(W)=\C_{(j_1,\ldots,j_{k+1})}(W).\]
\end{prop}
\begin{proof}
%sinceV1:  Rewrote next sentence thanks to Christos' comment.
It is sufficient to prove the proposition in the case where the two sequences agree except that $j'_i=j_{i+1}$ and $j'_{i+1}=j_i$ for some $i$.
Setting $x_0=1$ and $x_{k+1}=c$, a multichain $x_1\le x_2\le\cdots\le x_k$ is uniquely encoded by the sequence 
%referee:  Inserted parentheses on right side
\[(\delta_0,\ldots,\delta_{k})=(x_0^{-1}x_1,x_1^{-1}x_2,\ldots,x_k^{-1}x_{k+1}).\]
In~\cite{Armstrong}, this is called the \emph{delta sequence} of $x_1\le x_2\le\cdots\le x_k$.
A sequence of elements of~$L_W$ is a delta sequence for some multichain in~$L_W$ if and only if the absolute lengths of the elements of the sequence sum to $n=\rank(W)$ and the product, in order, of the sequence is~$c$.
A multichain is counted by $\C_{(j_1,\ldots,j_{k+1})}(W)$ if and only if its delta sequence has $\ell_T(\delta_{i-1})=j_i$ for all $i$.
Given a delta sequence $\delta=(\delta_0,\ldots,\delta_{k})$ with this property, define a new sequence $\delta'=(\delta'_0,\ldots,\delta'_{k})$ agreeing with $\delta$ except that $\delta'_{i-1}=\delta_{i-1}\delta_{i}\delta_{i-1}^{-1}$ and $\delta'_{i}=\delta_{i-1}$.
It is immediate that $\delta'$ is the delta sequence for a multichain counted by $\C_{(j'_1,\ldots,j'_{k+1})}(W)$ and furthermore that the map $\delta\mapsto\delta'$ defines a bijection between the two sets of chains.
\end{proof}

%referee:  changed k(t) to d(t) and k to d as appropriate through the end of this proof.
%referee:  reordered and revised the proof to better explain the conclusion.
We now prove the main theorem.
\begin{proof}[Proof of Theorem~\ref{jump chains}]
We continue to fix a particular Coxeter element $c=c_-c_+$.
Proposition~\ref{permute} implies that it is enough to consider the case where $i=k+1$.
A multichain counted by $\C_{(j_1,\ldots,j_k,1)}(W)$ consists of an element $x$ covered by~$c$ and a multichain in $[1,x]$ with rank-differences given by $(j_1,\ldots,j_k)$.
In light of Theorem~\ref{BWmain}(iii) and the prefix/postfix characterization of absolute order, an element $x$ is covered by~$c$ if and only if $x=ct$ for some $t\in T$.

For each $t\in T$, consider the orbit of $[1,ct]$ under the conjugation action of $\br{c_+,c_-}$.
Since conjugation by $c_-$ is an involutive isomorphism from $[1,c]$ to $[1,c^{-1}]$, all intervals in this orbit are isomorphic.
In light of Proposition~\ref{orbits}, we can define $d=d(t)$ to be the smallest $d\ge 0$ such that 
$c_+^{\br{d}}tc_+^{\br{-d}}=c_+^{\br{d+1}}tc_+^{\br{-d-1}}=s$ for some $s\in S$.
(Cf. the proof of \cite[Lemma~4.1]{ca2}.)
Let $s=s(t)=c_+^{\br{d}}tc_+^{\br{-d}}$.
If $d$ is even then $c_+sc_+=s$, or in other words, $s\in S_+$.
Thus in this case 
\[c_+^{\br{d}}(ct)c_+^{\br{-d}}=cs=\prod_{s'\in S_-}s'\prod_{s''\in (S_+\setminus\set{s})}s''.\]
If $d$ is odd then $s\in S_-$ and 
\[c_+^{\br{d}}(ct)c_+^{\br{-d}}=c^{-1}s=\prod_{s'\in S_+}s'\prod_{s''\in (S_-\setminus\set{s})}s''.\]
In either case $[1,ct]$ is isomorphic to $[1,c']$ where $c'$ is some Coxeter element for $W_{\br{s}}$.

By Lemma~\ref{para}, we conclude that $\C_{(j_1,\ldots,j_k,1)}(W)$ counts triples $(t,s,\mu)$ such that $t\in T$, $s=s(t)$ and $\mu$ is a multichain in $L_{W_{\br{s}}}$ with rank-differences given by $(j_1,\ldots,j_k)$.
Alternately, we can count such pairs by first specifying $s$ and applying Proposition~\ref{orbits}.
Each $s$ belonging to an orbit of size $\frac{h}{2}$ contributes $\frac{h}{2}\C_{(j_1,\ldots,j_{k})}(W_{\br{s}})$ to the count.
For each pair $s,s'$ belonging to an orbit of size $h$, the lattices $L_{W_{\br{s}}}$ and $L_{W_{\br{s'}}}$ are isomorphic, because they are isomorphic to conjugate intervals in $L_W$.
Together $s$ and $s'$ contribute $h\cdot \C_{(j_1,\ldots,j_{k})}(W_{\br{s}})=h\cdot \C_{(j_1,\ldots,j_{k})}(W_{\br{s'}})$ to the count.
\end{proof}

In general, the parabolic subgroups $W_{\br{s}}$ appearing in Theorem~\ref{jump chains} are not irreducible. 
When $W$ is reducible,~$L_W$ is a direct product.
Thus by basic chain-counting techniques we have

\begin{prop}
\label{reducible}
%sinceV1:  Rewrote next sentence thanks to Christos' comment.
If $W=W_1\times W_2$ with $\rank(W_1)=n_1$ then 
%referee:  Removed parentheses in RHS.
\[\C_{(j_1,\ldots,j_{k+1})}=\sum \C_{(j'_1,\ldots,j'_{k+1})}(W_1)\cdot \C_{\left(j_1-j'_1,\ldots,j_{k+1}-j'_{k+1}\right)}(W_2),\]
where the sum is over all sequences $(j'_1,\ldots,j'_{k+1})$ summing to $n_1$ and having $j'_i\le j_i$ for all $i\in[k+1]$.
\end{prop}

\begin{remark}\label{one formula}
%sinceV1:  deleted "recursive" before "formula."
Stembridge~\cite{Stembridge personal} pointed out that Theorems~\ref{jump chains} and Proposition~\ref{reducible} can be replaced by a single formula.
Let~$W$ be a finite Coxeter group, not necessarily irreducible.
For each $s\in S$, let $h_s$ denote the Coxeter number of the irreducible component of~$W$ containing~$s$.
%referee: changed _k to _l
One factors $c_+$ as $(c_+)_1(c_+)_2\cdots(c_+)_l$, where $(c_+)_i$ is in the $i$th irreducible component of~$W,$ and similarly for $c_-$.
For each $s\in S$, let $(c_+)_s$ be $(c_+)_i$ if~$s$ is in the $i$th irreducible component of~$W,$ and similarly $(c_-)_s$.
Replacing $c_\pm$ and $h$ by $(c_\pm)_s$ and~$h_s$ in the proof of Theorem~\ref{jump chains}, we obtain
%referee:  changed j_{k+i} to j_{k+1}
%frank:  inserted comma before \widehat
\[\C_{(j_1,j_2,\ldots,j_{k+1})}(W)\,=\,\frac{1}{2}\,\sum_{s\in S}h_s\cdot\C_{(j_1,j_2,\ldots,\widehat{j_i},\ldots,j_{k+1})}(W_{\br{s}}).\]
\end{remark}

\begin{remark}\label{ga sym}
We have seen that Proposition~\ref{orbits} describes a fundamental symmetry of~$L_W$.
In fact, the defining symmetry of cluster complexes also ultimately rests on Proposition~\ref{orbits}.
Specifically, \cite[Theorem~2.6]{ga}, which establishes the dihedral symmetry of the cluster complex, is a corollary of \cite[Proposition~2.5]{ga}, which in turn uses \cite[Lemma~2.1]{ga}, cited to \cite[Exercise~V.6.2]{Bourbaki}.
Proposition~\ref{orbits} is not stated explicitly in~\cite{Bourbaki}, but is the key to the results which can be applied to solve \cite[Exercise~V.6.2]{Bourbaki}.
And in fact, \cite[Theorem~2.6]{ga} is an easy corollary of Proposition~\ref{orbits}.
\end{remark}

\begin{remark}\label{nc rotate}
The concepts involved in the proof of Theorem~\ref{jump chains} shed light on another similarity between noncrossing partitions and clusters.
The definition of the cluster complex rests on a ``compatibility'' relation on certain roots.
%referee:  changed "sub root system" to "root subsystem"
A negative simple root $-\alpha$ is compatible with a root $\beta$ if and only if $\beta$ belongs to the parabolic root subsystem obtained by deleting $\alpha$.
The rest of the compatibility relation is defined by requiring that compatibility be invariant under the dihedral action of $\br{\tau_+,\tau_-}$, where $\tau_\pm$ is a modification of $c_\pm$.
A similar approach can be made to noncrossing partitions.
For $s\in S_-$ and $t\in T$, we have $st\in[1,c]$ if and only if $t\in W_{\br{s}}$.
When $s\in S_+$, we have $ts\in[1,c]$ if and only if $t\in W_{\br{s}}$.
This observation, together with the fact that the conjugation action of $\br{c_+,c_-}$ acts by automorphisms, completely determines~$L_W$.
\end{remark}

%sinceV1: changed title of this section.
\section{Applications of the main theorem}\label{app sec}
%sinceV1: Added an opening paragraph.  [Actually, didn't]
%Because of the hypothesis $j_i=1$ in Theorem~\ref{jump chains}, the theorem stops short of providing full recursive information on the number of chains in $L_W$ visiting specified ranks.
%In this section, we show how Theorem~\ref{jump chains} can be specialized to count several interesting classes of chains.

\subsection*{Maximal chains}\label{max sec}
%sinceV1:  changed "at least two" to "several"
The maximal chains in~$L_W$ are of particular interest for several reasons: 
For any finite Coxeter group, the maximal chains in~$L_W$ index the maximal faces in a CW-complex which is an Eilenberg-Maclane space (or ``K$(\pi,1)$'') for the associated Artin group~\cite{B3,BWKpi}. 
Furthermore, maximal chains of classical noncrossing partitions are in bijection with parking functions~\cite{StNC}.
%sinceV1:  added the following sentence:
In fact, the combinatorics of parking functions encodes rank-selected chain enumeration in the classical case \cite[Proposition~3.2]{StNC}.
Generalizations to the case $W=B_n$ have been studied in~\cite{Biane2,Hersh}.
%sinceV1:  added the following sentence:
The number of maximal chains of classical noncrossing partitions also coincides with the dimension of the ring of diagonal coinvariants.
(See, for example \cite[Theorem~4.2.4]{Haiman}.)

Let $\MC(W)$ denote the number of maximal chains in~$L_W$.
Since $\MC(W)=\C_{(1,1,\ldots,1)}(W)$, Theorem~\ref{jump chains} has the following corollary.
\begin{cor}
\label{max chains}
If $(W,S)$ is a finite irreducible Coxeter system then 
\[\MC(W)\,=\,\frac{h}{2}\,\sum_{s\in S}\MC(W_{\br{s}}).\]
\end{cor}

Proposition~\ref{reducible} becomes much simpler in this special case.
\begin{prop}
\label{max reducible}
If $W=W_1\times W_2$ with $\rank(W_1)=n_1$ and $\rank(W_2)=n_2$ then 
\[\MC(W)=\MC(W_1)\,\MC(W_2)\,\binom{n_1+n_2}{n_1}.\] 
\end{prop}

The recursions in Corollary~\ref{max chains} and Proposition~\ref{max reducible} can be solved to give formulas or values for each finite 
Coxeter group.
The results are tabulated below, followed by examples illustrating how they were obtained.

\vspace{.05 in}

\centerline{\small
\begin{tabular}{|c|c|c|c|c|c|c|c|c|c|}
\hline
$A_n$&$B_n$&$D_n$&$E_6$&$E_7$&$E_8$&$F_4$&$H_3$&$H_4$&$I_2(m)$\\\hline
&&&&&&&&&\\[-4mm]\hline
$\!(n+1)^{n-1}\!\!$&$n^n\!$&$2(n-1)^n\!$&$41472$&$1062882$&$37968750$&$432$&$50$&$1350$&$m$\\\hline
\end{tabular}}

\vspace{.05 in}

\begin{example}\label{easy}
There is one maximal chain in~$L_W$ when~$W$ has rank zero.
For $W=A_1$, since $h=2$ we have $\MC(A_1)=\frac{2}{2}\cdot 1=1$.
For $W=I_2(m)$ we have $h=m$, so $\MC(I_2(m))=\frac{m}{2}(1+1)=m$.
\end{example}

\begin{example}\label{h3}
For $W=H_3$, $h=10$ and the maximal parabolic subgroups are $I_2(5)$, $A_1\times A_1$ and $A_2$.
Corollary~\ref{max chains} and Proposition~\ref{max reducible} say that
\[\MC(H_3)=\frac{10}{2}\left(5+1\cdot 1\cdot\binom{2}{1}+3\right)=50.\]
\end{example}

\begin{example}\label{classical}
In each classical case, the formula for $\MC(W)$ is proved by induction, applying Abel's identity (see \cite{Abel}).
For $W=A_n$, the inductive step is
\[\sum_{i=0}^{n-1}\binom{n-1}{i}(i+1)^{i-1}(n-i)^{n-i-2}=2(n+1)^{n-2}.\]
This is proved by rewriting the binomial coefficient in the left side as a sum of two binomial coefficients, splitting into two sums, and reversing the order of summation in one of the sums.
The two summations are then identical, and by Abel's identity, each equals $(n+1)^{n-2}$.
For $W=B_n$, the inductive step is 
\[\sum_{i=0}^{n-1}\binom{n-1}{i}i^i(n-i)^{n-i-2}=n^{n-1},\]
which is proved by reversing the order of summation and applying Abel's identity.
For $W=D_n$, the inductive step is
\[\sum_{i=2}^{n-1}\binom{n-1}{i}(i-1)^i(n-i)^{n-i-2}=(n-1)^{n-1}-n^{n-2}.\]
This is proved by evaluating the left side from $i=0$ to $n-1$, reversing the order of summation, applying Abel's identity and then subtracting off the $i=0$ term.
\end{example}

The results tabulated above constitute a proof, without brute-force computer counting, of a uniform formula for $\MC(W)$ pointed out in \cite[Proposition~9]{Chapoton}.
\begin{theorem}\label{M}
If $W$ is a finite irreducible\footnote{When $W$ is reducible, the formula holds with $h^n$ replaced by $\prod_{s\in S}h_s$. (See Remark~\ref{one formula}.)} Coxeter group then 
\[\MC(W)\,=\,\frac{n!\,h^n}{|W|}.\]
\end{theorem}

%sinceSIAM:  rephrased first clause of next sentence.
As explained in~\cite{Chapoton}, Theorem~\ref{M} follows from a more general fact that has been verified type-by-type:
For irreducible~$W,$ the zeta polynomial of~$L_W$ is the Fuss-Catalan number $\Cat^{(m)}(W)$:
\begin{equation}
\label{Catm}
Z(L_W,m+1)=\Cat^{(m)}(W)=\prod_{i=1}^n\frac{mh+e_i+1}{e_i+1}.
\end{equation}
The numbers $e_i$ are fundamental numerical invariants called the exponents of~$W$.
The zeta polynomial $Z(P,q)$ of a poset $P$ counts, for each $q$, multichains $p_1\le p_2\le\cdots\le p_{q-1}$ in $P$.
See \cite[Section~3.11]{EC1} for details on zeta polynomials.
By \cite[Proposition~3.11.1]{EC1}, the leading term of $Z(P,q)$ is $(Mq^d)/d!$, where $d$ is the length of the longest chain in $P$ and $M$ is the number of chains of length $d$. 
Maximal chains in~$L_W$ have length $n=\rank(W)$, so the theorem follows by taking the coefficient of $m^n$ in Equation~(\ref{Catm}) and applying the fact that $|W|=\prod_{i=1}^n(e_i+1)$.

The proof of Theorem~\ref{M} by zeta polynomials suggests an alternate proof of Corollary~\ref{max chains}.
By \cite[Proposition~8.4]{gcccc}, $\Cat^{(m)}(W)$ also counts the facets of the generalized cluster complex associated to an irreducible~$W$.
A ``rotation'' of order $mh+2$ on the generalized cluster complex leads to a recursion \cite[Proposition~8.3]{gcccc} on $\Cat^{(m)}(W)$ and thus on zeta polynomials of~$L_W$:
\begin{equation}\label{zeta recursion}
Z(L_W,m+1)\,=\,\frac{mh+2}{2n}\,\sum_{s\in S}Z(L_{W_{\br{s}}},m+1).
\end{equation}
Corollary~\ref{max chains} arises by extracting the coefficient of $m^n$ in Equation~(\ref{zeta recursion}).
The juxtaposition of this alternate proof with the proof via Theorem~\ref{jump chains} is striking in that two different dihedral symmetries are related by passing to leading coefficients.

The proof of Corollary~\ref{max chains} via Equation~(\ref{zeta recursion}) can presumably be made uniform.
The fact that $Z(L_W,2)$ counts facets of the cluster complex (the case $m=1$ of the generalized cluster complex) is proven uniformly in~\cite{ABMW}, based on results of~\cite{BWlattice}.
Recently the results of~\cite{BWlattice} were extended~\cite{Tzanaki} to the case $m\ge 1$, and presumably the analogous extension of~\cite{ABMW} will eventually be undertaken.
The recursion counting facets of the generalized cluster complex was proven uniformly, except for \cite[Theorems~3.4 and~3.7]{gcccc}.
However, the extension of the results of~\cite{ABMW} to the case $m\ge 1$ can be expected to provide uniform proofs of \cite[Theorems~3.4 and~3.7]{gcccc}.
It should be stressed that a uniform proof of Equation~(\ref{Catm}), or even Theorem~\ref{M}, is completely lacking.
Indeed, a uniform proof of Equation~(\ref{Catm}) would specialize to a uniform proof that the number of elements of~$L_W$ equals $\Cat^{(1)}(W)$.
This is an important open problem in~$W$-Catalan combinatorics.

\subsection*{Reduced $T$-words}\label{TW sec}
Let $\TW_k(W)$ be the number of reduced $T$-words for elements of absolute length $k$ in~$L_W$.
These are chains $x_0<x_1<x_2<\cdots<x_k$ in~$L_W$ with $\ell_T(x_i)=i$ for each $i$, counted by $\C_{(1,\ldots,1,n-k)}(W)$.
Theorem~\ref{jump chains} implies:
\begin{cor}\label{TW}
If $(W,S)$ is a finite irreducible Coxeter system then 
\[\TW_k(W)\,=\,\frac{h}{2}\,\sum_{s\in S}\TW_{k-1}(W_{\br{s}}).\]
\end{cor}

Inspired by the alternate proof of Corollary~\ref{max chains}, we notice a connection between $\TW_k$ and the generalized cluster complex.
Let $f_k(W,m)$ be the number of $k$-vertex (i.e. $(k-1)$-dimensional) simplices in the generalized cluster complex associated to~$W$.
When~$W$ is irreducible, \cite[Proposition~8.3]{gcccc} says that
\begin{equation}\label{fk}
f_k(W,m)\,=\,\frac{mh+2}{2k}\,\sum_{s\in S}f_{k-1}(W_{\br{s}},m).
\end{equation}
Taking leading coefficients and interpreting the result as a recursion on $k!$ times the leading coefficient $\coef{m^k}{f_k(W,m)}$, we obtain a recursion identical to Corollary~(\ref{TW}).
Using Proposition~\ref{reducible} and another assertion of \cite[Proposition~8.3]{gcccc}, one easily checks that $\TW_k(W)$ also behaves like $k!\coef{m^k}{f_k(W,m)}$ when~$W$ is reducible and when $k=0$.
Thus by induction on the rank of $W,$ we have
\begin{theorem}\label{TW f}
For any finite Coxeter group $W,$
\[\TW_k(W)=k!\coef{m^k}{f_k(W,m)}.\]
\end{theorem}
In particular, (non-uniform) formulas for $\TW_k(W)$ can be obtained from \cite[Theorem~8.5]{gcccc}.
%sinceSIAM:  corrected/rephrased the following sentence thanks to Jon McCammond.
In light of the alternate proof of Theorem~\ref{M} via zeta polynomials, Theorem~\ref{TW f} suggests that $Z(L_W^{[0,k]},m+1)=f_k(W,m)$, where $L_W^{[0,k]}$ is the restriction of~$L_W$ to ranks $0,\ldots,k$.
However, this fails even in the smallest examples.

\subsection*{Edges}\label{edge sec}
Let $\E(W)$ be the number of edges in the Hasse diagram of~$L_W$.
Theorem~\ref{jump chains} specializes to a recursion on cover relations $x\covered y$ in~$L_W$ with $\ell_T(x)=i$.
The right side is a sum over $s\in S$ of the number of elements of $L_{W_{\br{s}}}$ at rank $i$.
Summing from $i=0$ to $i=n-1$ we obtain the following recursion on $\E(W)$.
\begin{cor}\label{E}
If $(W,S)$ is a finite irreducible Coxeter system then 
\[\E(W)\,=\,\frac{h}{2}\,\sum_{s\in S}\NC(W_{\br{s}}).\]
\end{cor}
Here $\NC(W)$ stands for the number of elements of~$L_W$.
Setting $m=1$ in Equation~(\ref{Catm}), we see that for $W$ irreducible,
\begin{equation}\label{NC}
\NC(W)=\Cat^{(1)}(W)=\prod_{i=1}^n\frac{h+e_i+1}{e_i+1}.
\end{equation}
Furthermore, $\Cat^{(1)}(W)$ counts facets of the cluster complex associated to $W$.
By a recursion \cite[Proposition~3.7]{ga} counting facets of the cluster complex,
\begin{equation}\label{NC recursion}
\NC(W)\,=\,\frac{h+2}{2n}\,\sum_{s\in S}\NC(W_{\br{s}}).
\end{equation}
Comparing Equation~(\ref{NC recursion}) with Corollary~\ref{E}, we obtain a uniform formula for the number of edges of~$L_W$.
\begin{theorem}\label{E formula}
If $(W,S)$ is a finite irreducible Coxeter system then 
\[\E(W)\,=\,\frac{nh}{h+2}\,\NC(W)\,=\,\frac{nh}{|W|}\prod_{i=2}^n(h+e_i+1).\]
\end{theorem}
The proof of Theorem~\ref{E formula} appears to be uniform, but is not, since there is no known uniform proof that $\NC(W)=\Cat^{(1)}(W)$.

\begin{remark}\label{edge remark}
Equation~(\ref{NC recursion}) can also be interpreted in terms of~$L_W$.
The equation can be rearranged to state that
\begin{equation}\label{rearrange}
\sum_{s\in S}\left(\NC(W)-\NC(W_{\br{s}})\right)=\frac{h}{2}\,\sum_{s\in S}\NC(W_{\br{s}}).
\end{equation}
%sinceV1:  changed text below (thanks to Armstrong's comment).
The right side of the equation is $\E(W)$.
The left side of Equation~(\ref{rearrange}) counts pairs $(s,x)\in S\times L_W$ with $x\not\in W_{\br{s}}$.
(Cf. the proof of Lemma~\ref{para}.)
A comparison of the expression for $\E(W_1\times W_2)$ arising from Proposition~\ref{reducible} with the left side of Equation~(\ref{rearrange}) in the case $W=W_1\times W_2$ shows that, even in the reducible case, $E(W)$ counts pairs $(s,x)$ as above.
We have no explanation, within the combinatorics of noncrossing partitions, for the coincidence between these two counts.
\end{remark}

\subsection*{Saturated chains}\label{sat sec}
Theorem~\ref{jump chains} specializes to a recursion on saturated chains $x_0\covered x_1\covered\cdots\covered x_k$ with $\ell_T(x_0)=i$.
Summing over all possible $i$, we obtain the following recursion on the total number $\SC_k(W)$ of saturated chains of length $k$ (i.e.\ $k+1$ elements) in~$L_W$.
\begin{cor}\label{SC}
If $(W,S)$ is a finite irreducible Coxeter system and $k>0$ then
\[\SC_k(W)\,=\,\frac{h}{2}\,\sum_{s\in S}\SC_{k-1}(W_{\br{s}}).\]
\end{cor}

Corollaries~\ref{max chains} and~\ref{E} are special cases of Corollary~\ref{SC} which lead to uniform formulas, and Equation~(\ref{NC}) is a uniform formula for $\SC_0(W)$.
One more special case leads to a uniform formula.
When $k=n-1$, the recursion of Corollary~\ref{SC} is solved in essentially the same manner as the recursion of Corollary~\ref{max chains}, to obtain, for $W$ irreducible,
\begin{equation}\label{almost}
\SC_{n-1}\,=\,\frac{2\,n!\,h^n}{|W|}.
\end{equation}
This is also a trivial corollary of Theorem~\ref{M}, since~$L_W$ has a unique minimal and maximal element.
Naturally, one seeks a formula for $\SC_k$ that generalizes Theorems~\ref{M} and~\ref{E formula} and Equations~(\ref{NC}) and~(\ref{almost}).
However, the obvious generalization
\begin{equation}\label{obvious}
n(n-1)\cdots(n-k+1)\,\,\frac{h^k}{|W|}\,\prod_{i=k+1}^n(h+e_i+1)
\end{equation}
does not work beyond the cases $k\in\set{0,1,n-1,n}$.
It appears to fail in a way that is exactly analogous to the situation for $f$-numbers and $h$-numbers of cluster complexes, as explained in Section~\ref{m sec}.

\section{$m$-Divisible noncrossing partitions}\label{m sec}
This section is an extended remark about generalizing the results of this paper to the poset of \emph{$m$-divisible noncrossing partitions}, as defined by Armstrong in~\cite{Armstrong}.
We call this poset $L^{(m)}_W$.
 Armstrong names it $\NC_{(m)}(W)$ and also considers the dual poset, under the name $\NC^{(m)}(W)$.
The $m$-divisible noncrossing partitions are the delta sequences of $m$-element multichains in~$L_W$.
(See the proof of Proposition~\ref{permute}.)
In particular, $L_W^{(1)}=L_W$.

Let $\delta=(\delta_0,\ldots,\delta_{k})$ and $\delta'=(\delta'_0,\ldots,\delta'_{k})$ be delta sequences of multichains in~$L_W$.
The definition of $L_W^{(m)}$ sets $\delta\le\delta'$ if and only if $\delta_i\le\delta'_i$ for all $i=1,\ldots,k$.
The poset $L^{(m)}_W$ is a graded meet-semilattice \cite[Theorem~3.4.4]{Armstrong} with rank function $\sum_{i=1}^n\ell_T(\delta_i)$.
Let $\C^{(m)}_{(j_1,j_2,\ldots,j_{k+1})}(W)$ count multichains $x_1\le x_2\le\cdots\le x_k$ in $L^{(m)}_W$ with rank-differences given by $(j_1,j_2,\ldots,j_{k+1})$.
Theorem~\ref{jump chains} generalizes as follows:
\begin{theorem}\label{m jump chains}
If $(W,S)$ is a finite irreducible Coxeter system and $j_i=1$ then
%sinceSIAM: changed j_{k+i} to j_{k+1}
%frank:  inserted comma before \widehat
\[\C^{(m)}_{(j_1,j_2,\ldots,j_{k+1})}(W)\,=\,\frac{mh}{2}\,\sum_{s\in S}\C^{(m)}_{(j_1,j_2,\ldots,\widehat{j_i},\ldots,j_{k+1})}(W_{\br{s}}).\]
\end{theorem}

Theorem~\ref{m jump chains} is proved by a straightforward but notationally cumbersome generalization of the proof of Theorem~\ref{jump chains}, replacing the action of $\br{c_+,c_-}$ with the action of $\br{L^*,R^*}$, as defined in \cite[Section~3.4.6]{Armstrong}.
We omit the details.

Theorem~\ref{m jump chains} implies a formula for the number $\MC^{(m)}$ of maximal chains in $L^{(m)}_W$:
\begin{equation}\label{m max chains}
\MC^{(m)}(W)\,=\,\frac{mh}{2}\,\sum_{s\in S}\MC^{(m)}(W_{\br{s}}).
\end{equation}
The powers of $m$ in the formula can be factored out trivially, so that this recursion is solved exactly as in the case $m=1$.
Thus for $W$ irreducible we have $\MC^{(m)}(W)=\frac{n!(mh)^n}{|W|}$, as was pointed out earlier by Armstrong \cite[Corollary~3.6.10]{Armstrong}.

The analog of Corollary~\ref{TW} behaves similarly: powers of $m$ can be factored out.
Thus saturated chains $x_0<x_1<x_2<\cdots<x_k$ such that $x_0$ is the unique minimal element of $L^{(m)}_W$ are counted by $m^k\TW_k(W)$.
It is an easy exercise, given some familiarity with $L^{(m)}_W$, to construct an $(m^k)$-to-1 map from these ``lower-saturated'' chains in $L^{(m)}_W$ to lower-saturated chains in~$L_W$. 
There are also $m$-analogs of the formulas given in Corollary~\ref{E}, Theorem~\ref{E formula} and Remark~\ref{edge remark}, for the number $E^{(m)}(W)$ of edges in $L^{(m)}_W$.
In each case, with $h$ replaced by $mh$ and with superscripts ``$(m)$'' in the appropriate places, the analogous proof works.

Corollary~\ref{SC} generalizes, with similar modifications, to a recursion counting saturated chains in $L^{(m)}_W$.
The resulting type-by-type formulas for $\SC_k^{(m)}(W)$ are polynomials in $m$.
Factoring these into irreducible factors reveals multiplicative formulas which seem to be loosely based on Equation~(\ref{obvious}), with $h$ replaced by $mh$ throughout.
These factorizations exhibit the odd phenomena first observed in formulas \cite[Theorems~8.5 and~10.2]{gcccc} for the $f$- and $h$-numbers of generalized cluster complexes, including ``levels'' of exponents and single ``mysterious'' factors.
The formulas for $\SC_k^{(m)}(W)$ appear to be badly behaved exactly when $f_k(W,m)$ and $h_k(W,m)$ are badly behaved, and the bad behaviors take the same form.

%sinceV1:  Added acknowledgments
%sinceSIAM:  Added Jon McC to the acknowledgments.
\section{Acknowledgments} 
I thank Drew Armstrong, Christos Athanasiadis, Jon McCammond and John Stembridge for helpful comments on earlier versions of this paper.

\end{document}